\documentclass[conference]{IEEEtran}

\usepackage{color}
\usepackage{epsfig}
\usepackage{subfig}
\usepackage{mathrsfs}
\usepackage{booktabs}
\usepackage{graphicx}
\usepackage{epstopdf}
\usepackage{multirow}
\usepackage{amsmath,amssymb}
\usepackage{amsthm}
\usepackage{fancyhdr}
\usepackage{tabularx}
\usepackage{xfrac}

\ifCLASSINFOpdf
\else
\fi

\newtheorem{prop}{Proposition}[section]

\newtheorem{remark}{Remark}[section]

\renewcommand\IEEEkeywordsname{Keywords}

\begin{document}

\title{Comparing Different Models for Investigating Cascading Failures in Power Systems}

\author{
\IEEEauthorblockN{Chao Zhai, Hehong Zhang, Gaoxi Xiao}
\IEEEauthorblockA{Institute of Catastrophe Risk Management\\
School of Electrical and Electronic Engineering \\
Nanyang Technological University \\
50 Nanyang Avenue, Singapore 639798.\\
Future Resilient Systems, Singapore-ETH Centre\\
1 Create Way, CREATE Tower, Singapore 138602.\\
Email: EGXXiao@ntu.edu.sg}
\and
\IEEEauthorblockN{Tso-Chien~Pan}
\IEEEauthorblockA{Institute of Catastrophe Risk Management\\
School of Civil and Environmental Engineering \\
Nanyang Technological University \\
50 Nanyang Avenue, Singapore 639798.\\
Future Resilient Systems, Singapore-ETH Centre\\
1 Create Way, CREATE Tower, Singapore 138602.}
}

\maketitle

\begin{abstract}
This paper centers on the comparison of three different models that describe cascading failures of power systems. Specifically, these models are different in characterizing the physical properties of power networks and computing the branch power flow. Optimal control approach is applied on these models to identify the critical disturbances that result in the worst-case cascading failures of power networks. Then we compare these models by analyzing the critical disturbances and cascading processes. Significantly, comparison results on IEEE 9 bus system demonstrate that physical and electrical properties of power networks play a crucial role in the evolution of cascading failures, and it is necessary to take into account these properties appropriately while applying the model in the analysis of cascading blackout.
\end{abstract}


\IEEEpeerreviewmaketitle

\section{Introduction}
Cascading blackout of power networks usually affects large areas of numerous people and results in huge economic losses. For instance, the blackout in United States and Canada in August 2003 was caused by the operation of protective relay to sever the overloading branches and the inadequate situational awareness of operators \cite{empg04}, affecting an area of 50 million people causing a loss of more than 4 billion U.S. dollars. In November 2006, a few European countries experienced a severe blackout, which was triggered by the tripping of several high-voltage lines in Northern Germany and resulted in the outage of power supply for 15 million residents \cite{utce07}. Thus, it is of great significance to come up with an appropriate model in order to identify the critical initial disturbances, which may help to eliminate the cascading blackout of power systems in advance.

Prior to making the model, it is necessary to figure out the cause and process of cascading blackout in power systems. Normally, a cascading blackout is initiated by one contingency event of component outage and subsequent operator errors, which brings about a sequence of component outages due to the branch overloading \cite{hine16,dob07,vai12}. During the cascading process, each cascading step is considered as a topological change of power networks. To the best of our knowledge, there have been mainly three types of models characterizing cascading failures of power systems \cite{hine16}. To be specific, the first type of models only describes the topological properties of power networks and neglects the underlying laws of physics and the principles of electrotechnics \cite{yuy16,hine10}, while the second one takes into account the quasi-steady-state of power systems and computes the power flow by solving the direct current (DC) or the alternate current (AC) power flow equations \cite{sol14,cz17,hhz17}. In addition, the last one aims to investigate the emergence of cascading failures via dynamical modeling of power system components \cite{cate84}. Statistically, the topology based complex network models behave like the DC power flow model under intentional attacks \cite{ou14}. Nevertheless, \cite{pag13} surveys the approaches of complex networks analysis in power grids and points out the necessity to incorporate the physical and electrical properties. This work attempts to compare different models and determine the scope of their applications.

In the past decades, coordination and control of multi-agent systems has attracted great interests of researchers in various fields \cite{cz12,htz11,zhai16}. Multi-agent system approaches are applied to power system control and protection since a bus in power grids can be regarded as a smart agent able to communicate and interact with its neighbors. \cite{bab16} proposes an adaptive multi-agent system algorithm to prevent cascading failures. In this paper, we treat each bus as an agent that is able to transmit, receive and consume power in power grids, and optimal control theory is employed to identify the critical disturbances that give rise to worst-case cascading failures of power networks. The proposed models are compared by analyzing the critical disturbances and cascading processes.

The remainder of this paper is organized as follows. Section \ref{sec:prob} presents the three different models and the optimal control formulation of identifying disturbances. Section \ref{sec:the} provides theoretical results for the optimal control problem, followed by numerical simulations and comparisons in Section \ref{sec:sim}. Finally, we conclude the paper and discuss future work in Section \ref{sec:con}.

\section{Problem Statement}\label{sec:prob}

\begin{figure*}
\scalebox{0.06}[0.06]{\includegraphics{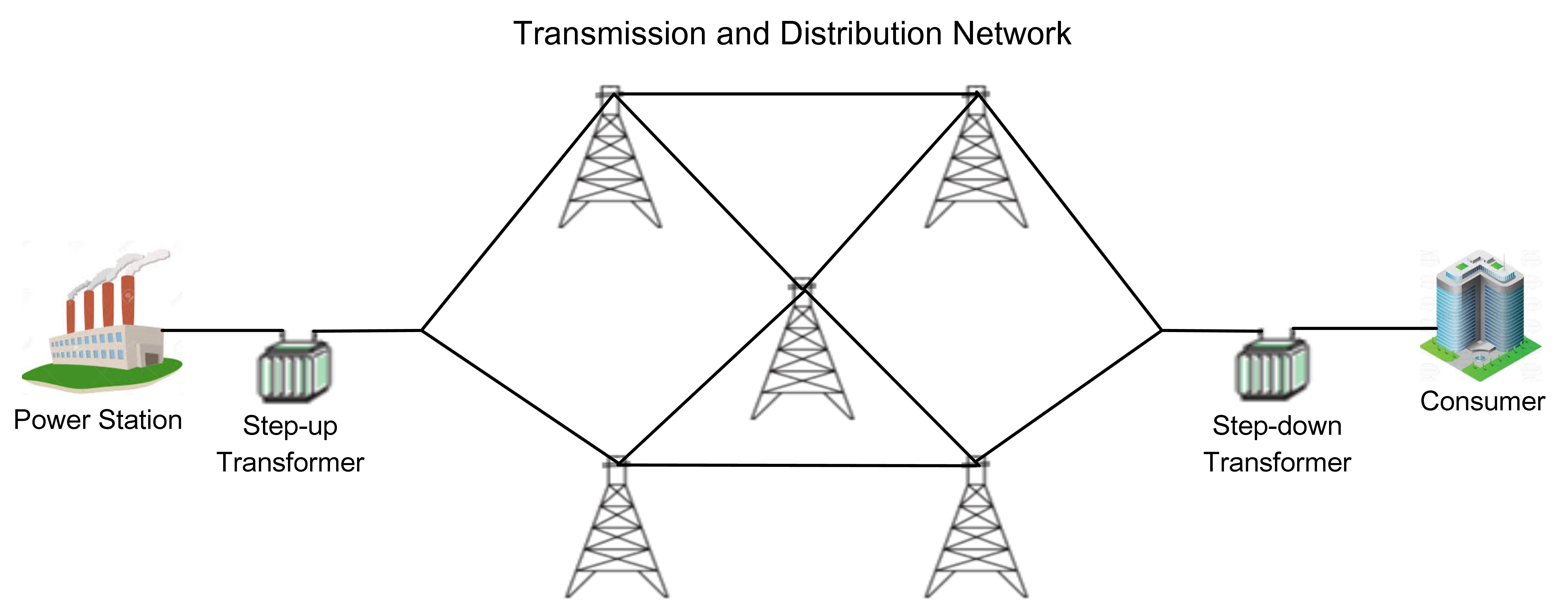}}\centering
\caption{\label{pnet} Schematic diagram on the components of power systems.}
\end{figure*}

An electrical power system is normally composed of power stations that generate electric power, transformers which raise or lower the voltages, power transmission and distribution networks and consumers to use electric power (see Fig.~\ref{pnet}). This paper focuses on the comparison of network models characterizing the dynamical evolution of the transmission and distribution networks in power systems. Fig.~\ref{flow} presents the cascading process of power systems in network models. Specifically, the disruptive disturbances change the branch admittance or the connection level, which leads to the overloading of some branches. Then the overloaded branches are removed from the network to describe the operation of circuit breakers in power grids, which reconfigures the network topology. As a result, the change of network topology leads to the imbalance of power flow once again. In short, it takes turns to compute the power flow and update the network topology, which describes the cascading process of power systems in practice. The above process terminates once the power flow is balanced without further evolution of network topology.

Our goal is to compare three different models by identifying disruptive disturbances on vulnerable branches that result in worst-case cascading failures and exploring the effects of physical properties of power systems on the evolution of cascading blackouts.

Actually, the three models share the same mechanism of severing the overloading branches. To facilitate the theoretical analysis, we design the following function to describe the state shift of overloading branches
\begin{equation}\label{step}
g(P_{ij},c_{ij})=\left\{
                \begin{array}{ll}
                  0, & \hbox{$|P_{ij}|\geq \sqrt{c_{ij}^2+\frac{\pi}{2\sigma}}$;} \\
                  1, & \hbox{$|P_{ij}|\leq\sqrt{c_{ij}^2-\frac{\pi}{2\sigma}}$;} \\
                  \frac{1-\sin\sigma (P_{ij}^2-c_{ij}^2)}{2}, & \hbox{otherwise.}
                \end{array}
              \right.
\end{equation}
where $c_{ij}$ denotes the power threshold of transmission line connecting Bus $i$ to Bus $j$, and $P_{ij}$ refers to the transmission power on this branch. It is worth pointing out that the function $g(P_{ij},c_{ij})$ approximates to step function as the tunable parameter $\sigma$ gets close to the positive infinity.

Essentially, the three models are different in the computation of power flow. Next, we present their mathematical expressions as follows.

\subsection{Complex Network Model (CNM)}
Without the consideration of physical properties, the complex networks model of power systems is given by
\begin{equation}\label{complex}
P=A^Tdiag(S^k)P^k_e
\end{equation}
where $P\in R^{n_b}$ denotes the vector of injected power on a total of $n_b$ buses and $A$ refers to the branch-bus incidence matrix \cite{stag68}. $P^k_e=(P^k_{ij})\in R^n$ is the vector of transmission power on a total of $n$ branches at the $k$-th cascading step. $S^k$ is the state vector on the connection level of branches. The operation $diag(x)$ obtains a square diagonal matrix with the elements of vector $x$ on the main diagonal. Since the matrix $A^T$ could be nonsingular or a non-square matrix, a least square solution to (\ref{complex}) is given as
\begin{equation}\label{comp_sol}
P^k_e=\left[A^Tdiag(S^k)\right]^{+}P
\end{equation}
where $[A^Tdiag(S^k)]^{+}$ stands for the Moore-Penrose pseudoinverse of $A^Tdiag(S^k)$ \cite{mp95}.

\subsection{DC-Based Network Model (DCM)}
For high-voltage transmission networks, the DC power flow equation is employed to compute the power flow \cite{stot09}.
\begin{equation}\label{dc_power}
P=A^Tdiag(Y_p^k)A\theta^k, \quad k=0,1,2,...,m-1
\end{equation}
where $Y_p^k$ denotes the vector of branch admittance at the $k$-th cascading step, and $\theta^k$ represents the vector of voltage angle on each bus. Moreover, the solution to (\ref{dc_power}) is expressed as
$$
\theta^k=(A^Tdiag(Y_p^k)A)^{-1^*}P
$$
with the operator $-1^*$ being defined as a type of pseudoinverse in \cite{cz17}. According to Lemma 3.2 in \cite{cz17}, the vector of transmission power on $n$ branches at the $k$-th cascading step is given by
\begin{equation}\label{pijk}
P^k_{e}=diag(Y^k_p)A(A^T diag(Y^k_p)A)^{-1^*}P
\end{equation}
where $P^k_e=(P^k_{ij})\in R^n$ with $i,j\in I_{n_b}=\{1,2,...,n_b\}$.

\subsection{AC-Based Network Model (ACM)}
The AC power flow equation allows for both active power and reactive power in power networks, which provides the most accurate description of practical power networks despite high computation cost. For Bus $i\in I_{n_b}$, the AC power flow equation is expressed as
\begin{equation}\label{acpower}
\begin{split}
P_i&=\sum_{j=1}^{n_b}|V^k_i|\cdot|V^k_j|(G^k_{ij}cos\theta^k_{ij}+B^k_{ij}sin\theta^k_{ij})\\
Q_i&=\sum_{j=1}^{n_b}|V^k_i|\cdot|V^k_j|(G^k_{ij}sin\theta^k_{ij}-B^k_{ij}cos\theta^k_{ij})\\
\end{split}
\end{equation}
where $P_{i}$ and $Q_i$ are the net active power and reactive power injected at Bus $i$, respectively. Likewise, $|V^k_i|$ and $|V^k_j|$ denote the voltage magnitude on Bus $i$ and Bus $j$ at the $k$-th cascading step, respectively. $G^k_{ij}$ and $B^k_{ij}$ are the real part and the imaginary part of the element in the bus admittance matrix $A^Tdiag(Y^k_p)A$ corresponding to the $i$-th row and $j$-th column, respectively. $\theta^k _{ij}=\theta^k _{i}-\theta^k _{j}$ is the difference in voltage angle between Bus $i$ and Bus $j$ \cite{gra94}. Unfortunately, there are no analytical solutions to equation (\ref{acpower}). Thus we can only estimate its solution via numerical methods.
\begin{remark}
The number of unknowns in the AC power flow equation depends on the total number of buses in power networks and the distribution of load buses and generator buses. For load buses, the injected active power and reactive power are known while the injected active power and the voltage magnitude are available for generator buses.
\end{remark}

\begin{figure}[t!]
\scalebox{0.06}[0.06]{\includegraphics{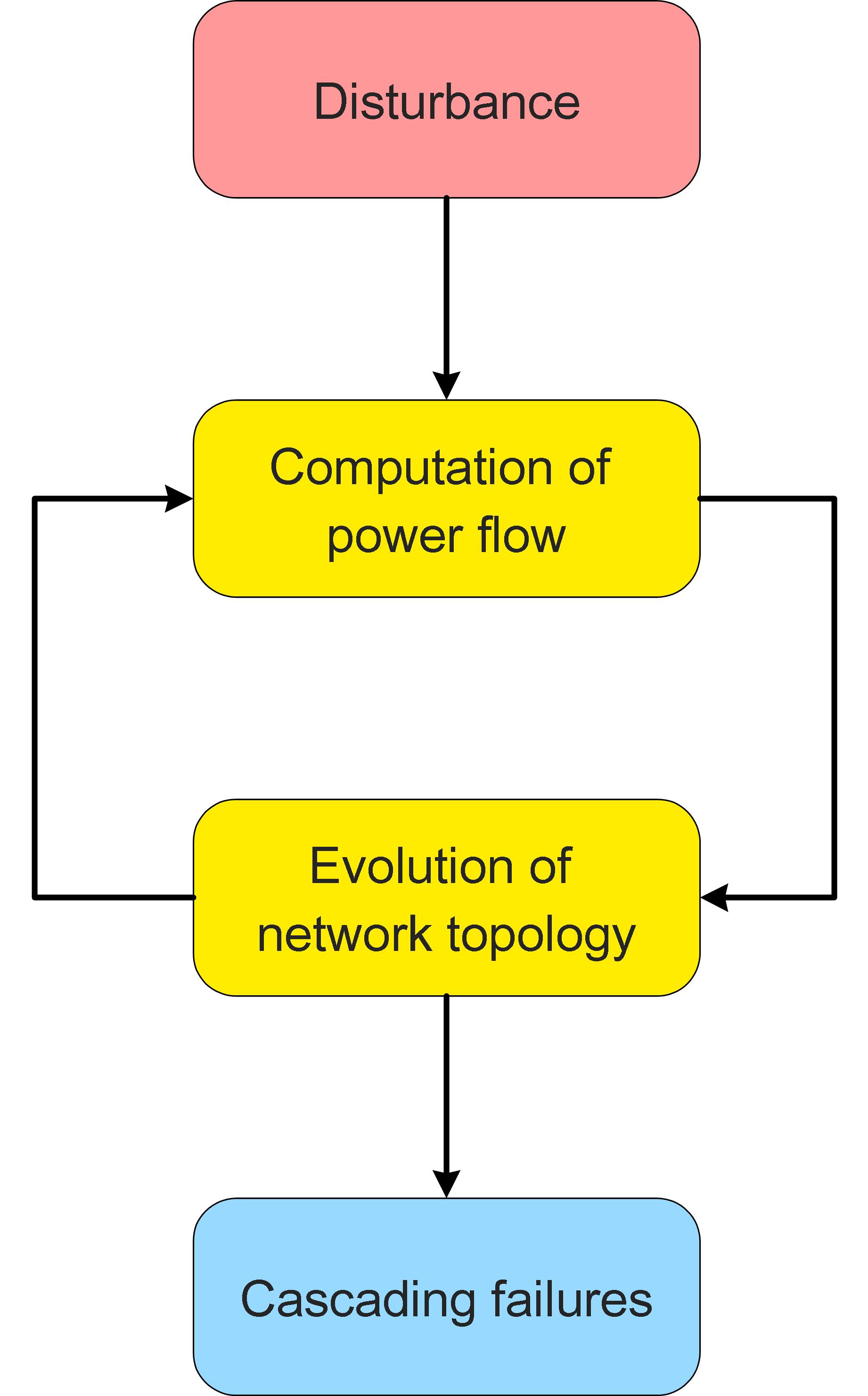}}\centering
\caption{\label{flow} Cascading process of power systems in three models.}
\end{figure}

\subsection{Optimization Formulation}
Then we present the state equation of transmission lines as follows
\begin{equation}\label{cas_model}
X^{k+1}=\mathcal{G}(P_e^k)X^k+U^k, , \quad k=0,1,...,m-1
\end{equation}
where $U^k$ represents the control input or disturbance on the given branches, and $X^k$ denotes the state vector of branches at the $k$-th cascading step. Specifically, $X^k$ describes the branch state of connection in the CNM while it characterizes the branch admittance in the DCM  and ACM. In short, $X^k$ can be expressed as
$$
X^k=\left\{
  \begin{array}{ll}
    S^k, & \hbox{CNM;} \\
    Y^k_p, & \hbox{DCM or ACM.}
  \end{array}
\right.
$$
In particular, the state matrix of transmission lines  $\mathcal{G}(P_{e}^{k})$ at the $k$-th cascading step is given by
$$
\mathcal{G}(P_e^{k})=diag\left(
                          \begin{array}{c}
                            g(P_{i_1j_1}^{k},c_{i_1j_1}) \\
                            g(P_{i_2j_2}^{k},c_{i_2j_2}) \\
                            . \\
                            g(P_{i_nj_n}^{k},c_{i_nj_n}) \\
                          \end{array}
                        \right)\in R^{n \times n}
$$

The identification of initial disturbances that cause the worst cascading failures of power systems is formulated as the following optimal control problem.
\begin{equation}\label{cost}
\min_{U^k}J(X^m,U^k)
\end{equation}
with the cost function
\begin{equation}\label{cost_fun}
J(X^m,U^k)=\|X^m\|^2+\epsilon\sum_{k=0}^{m-1}\frac{\|U^k\|^2}{\max\{0,1-k\}}
\end{equation}
where $\epsilon$ is a positive weight, and $\|\cdot\|$ represents the 2-norm. $m$ is the total number of cascading steps. The proposed cost function includes two terms. Specifically, the first term $\|X^m\|^2$ is differentiable with respect to $X^m$, and it quantifies the connectivity of power networks at the final cascading step, and the second term characterizes the control energy or disturbance strength at the initial step. In addition, the parameter $\epsilon$ is set small enough so that the first term dominates in the cost function. The objective is to minimize $\|X^m\|^2$ by adding the appropriate control input or initial disturbance $U^0$ on the selected branches of power networks.

\section{Theoretical Analysis}\label{sec:the}
In this section, we present theoretical results on optimal control problem (\ref{cost}). For the discrete time nonlinear system, optimal control theory provides the necessary conditions of minimizing the cost function.
\begin{prop}\label{disopt}
For the discrete time optimal control problem
$$
\min_{U^k}J(X^k,U^k)
$$
with the state equation
$$
X^{k+1}=F(X^k,U^k), \quad k=0,1,...,m-1
$$
and the cost function
$$
J(X^k,U^k)=\Phi(X^m)+\sum_{k=0}^{m-1}L(X^k,U^k),
$$
the necessary conditions for the optimal control input $U^{k*}$ are given as follows
\begin{enumerate}
  \item $X^{k+1}=F(X^k,U^k)$
  \item $\lambda_k=(\frac{\partial F}{\partial X^k})^T\lambda_{k+1}+\frac{\partial L}{\partial X^k}$
  \item $(\frac{\partial F}{\partial U^k})^T\lambda_{k+1}+\frac{\partial L}{\partial U^k}=0$
  \item $\lambda_m=\frac{\partial \Phi}{\partial X^m}$
\end{enumerate}
where $\lambda_k$ denotes the costate variable.
\end{prop}

\begin{proof}
It is a special case ($i.e.$, time invariant case) of the optimal control for the time-varying discrete time nonlinear system in \cite{fran95}. Hence the proof is omitted.
\end{proof}
By applying Proposition \ref{disopt} to the optimal control problem (\ref{cost}), we obtain the necessary conditions for identifying the initial disturbance of power systems with state equation (\ref{cas_model}).
\begin{prop}\label{sysAE}
The necessary condition for the optimal control problem (\ref{cost}) is given by solving the following system of algebraic equations.
\begin{equation}\label{con_sys}
X^{k+1}=\mathcal{G}(P^k_{e})X^k+U^k, \quad k=0,1,...,m-1
\end{equation}
and the control input $U^k$ is expressed as
$$
U^k=\left\{
      \begin{array}{ll}
        -\frac{1}{\epsilon}\prod_{s=0}^{m-2}\frac{\partial X^{m-s}}{\partial X^{m-s-1}}\cdot X^m, & \hbox{$k=0$;} \\
        \mathbf{0}_n, & \hbox{$k\geq1$.}
      \end{array}
    \right.
$$
where $\mathbf{0}_n=(0,0,...,0)^T\in R^n$
\end{prop}

\begin{proof}
According to Proposition \ref{disopt}, the necessary conditions for the optimal control problem (\ref{cost}) can be determined as
\begin{equation}\label{cond_st}
X^{k+1}=\mathcal{G}(P_{e}^k)\cdot X^{k}+U^k
\end{equation}

\begin{equation}\label{cond_u}
\left(\frac{\partial X^{k+1}}{\partial U^k}\right)^T\lambda_{k+1}+\frac{\epsilon }{\max\{0,1-k\}}\cdot\frac{\partial\|U^k\|^2}{\partial U^k}=0
\end{equation}

\begin{equation}\label{cond_y}
\lambda_k=\left(\frac{\partial X^{k+1}}{\partial X^k}\right)^T\lambda_{k+1}+\frac{\epsilon }{\max\{0,1-k\}}\cdot\frac{\partial\|U^k\|^2}{\partial X^k}
\end{equation}

\begin{equation}\label{cond_final}
\lambda_m=\frac{\partial \|X^m\|^2}{\partial X^m}
\end{equation}
Thus, solving (\ref{cond_u}) leads to
\begin{equation}\label{cond_uk}
U^k=-\frac{\lambda_{k+1}}{2\epsilon}\max\{0,1-k\}
\end{equation}
and simplifying (\ref{cond_y}) yields
\begin{equation}\label{cond_lamb}
\lambda_k=\left(\frac{\partial X^{k+1}}{\partial X^k}\right)^T\lambda_{k+1}
\end{equation}
with the final condition $\lambda_m=2X^m$ being derived from (\ref{cond_final}).
Then we obtain
\begin{equation}\label{cond_lambd}
\lambda_{k+1}=2\prod_{s=0}^{m-k-2}\frac{\partial X^{m-s}}{\partial X^{m-s-1}}\cdot X^m.
\end{equation}
Combining (\ref{cond_uk}) and (\ref{cond_lambd}), we obtain
\begin{equation}\label{cond_ukh}
U^k=-\frac{\max\{0,1-k\}}{\epsilon}\prod_{s=0}^{m-k-2}\frac{\partial X^{m-s}}{\partial X^{m-s-1}}\cdot X^m
\end{equation}
which is equivalent to
$$
U^k=\left\{
      \begin{array}{ll}
        -\frac{1}{\epsilon}\prod_{s=0}^{m-2}\frac{\partial X^{m-s}}{\partial X^{m-s-1}}\cdot X^m, & \hbox{$k=0$;} \\
        \mathbf{0}_n, & \hbox{$k\geq1$.}
      \end{array}
    \right.
$$
Substituting (\ref{cond_ukh}) into (\ref{cond_st}) yields
\begin{equation}
\begin{split}
X^{k+1}=\mathcal{G}(P^k_{e})X^k&-\frac{\max\{0,1-k\}}{\epsilon}\prod_{s=0}^{m-k-2}\frac{\partial X^{m-s}}{\partial X^{m-s-1}}\cdot X^m \\
&~k=0,1,...,m-1
\end{split}
\end{equation}
which is the integrated mathematical expression of necessary conditions (\ref{cond_st}), (\ref{cond_u}), (\ref{cond_y}) and (\ref{cond_final}) for the optimal control problem (\ref{cost}). The proof is thus completed.
\end{proof}

\begin{figure}
\scalebox{0.65}[0.65]{\includegraphics{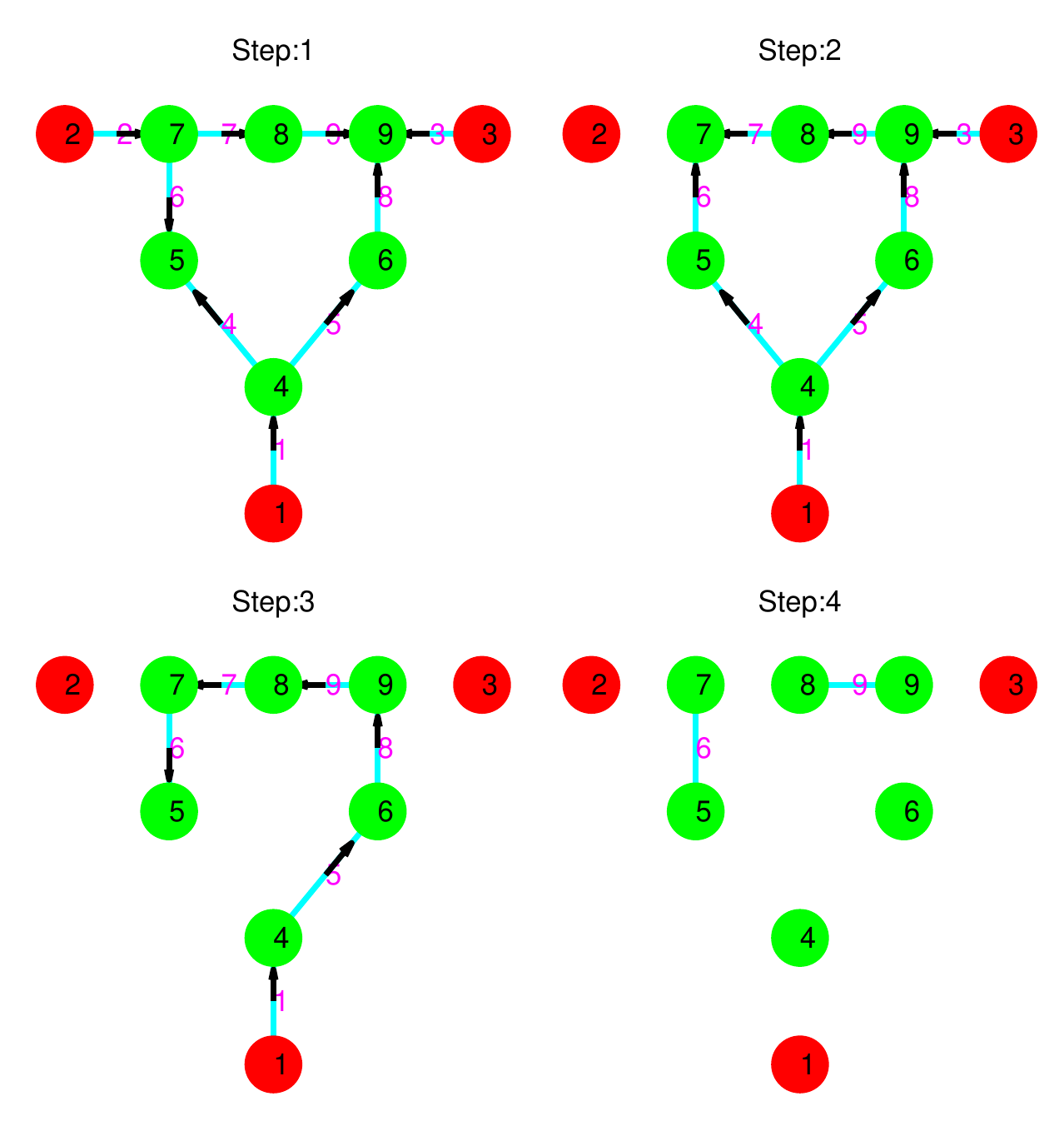}}\centering
\caption{\label{cnm} Cascading process of the IEEE 9 bus system based on the CNM. Red balls denote the generator buses, and green ones refer to the load buses. Bus identity (ID) numbers and branch ID numbers are marked as well. The arrows represent the power flow on each branch. A branch is severed once its transmission power exceeds the given threshold. The arrow disappears if there is no power transmission on the branch.}
\end{figure}

\begin{remark}
Optimal control theory enables us to obtain the initial disturbances leading to worst-case cascading failures of power networks described by the CNM and the DCM. Nevertheless, it does not apply to the ACM for the time being since the closed-form solutions to the AC power flow equation are not available.
\end{remark}

\section{Numerical Simulations}\label{sec:sim}

In this section, we present the numerical solution to the system of algebraic equations (\ref{con_sys}) for both CNM and DCM of IEEE 9 bus system. For the ACM, we analyze its cascading process by adding the computed disturbance based on the DCM. As is known, IEEE 9 bus system is composed of 9 buses (3 generator buses and 6 load buses) and 9 branches \cite{zim11}. Per unit values are adopted with the base value of power $100$ MVA in numerical simulations. Other parameters for the models of power networks are given as $\sigma=5\times10^4$, $\epsilon=10^{-4}$ and $m=9$. The vector of power threshold on each branch is $c=(c_{ij})=(1,1.8,1,0.6,0.5,1,1,1,1)$. In addition, the solver ``fsolve" in Matlab is adopted to solve the system of algebraic equations (\ref{con_sys}).

\subsection{Cascading Process}
The parameter setting of the CNM is the same as the above except that $\epsilon=6\times10^{-8}$, indicating that more efforts are taken to minimize the first term of cost function (\ref{cost_fun}). We choose Branch 2 as the target to add the initial disturbance with the value of $-1.22$, which is obtained by solving the system of algebraic equations (\ref{con_sys}). Initially, this computed disturbance does not sever Branch 2, but it immediately leads to the outage of Branch 2 due to the overloading of branch power flow. In Fig.~\ref{cnm},  Step 1 describes the normal state of power systems without any disturbances.  The outage of Branch 2 at Step 2 triggers the chain reactions at Step 3 and Step 4. Finally, the power network stops evolving after Step 4 and ends up with 2 connected branches and no transmission of power flow.

\begin{figure}[t!]
\scalebox{0.68}[0.68]{\includegraphics{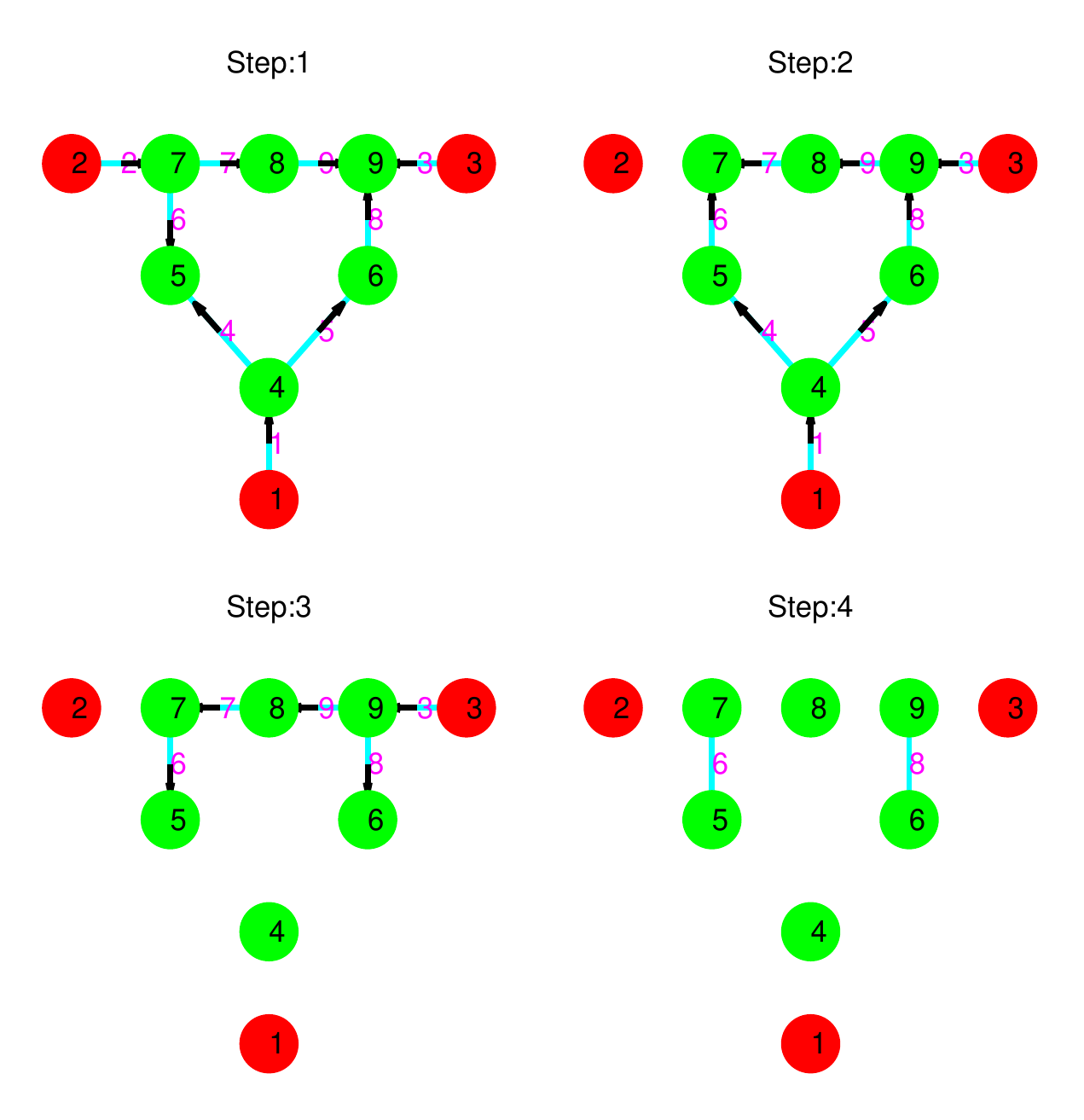}}\centering
\caption{\label{dcm} Cascading process of the IEEE 9 bus system based on the DCM.}
\end{figure}

For the DCM, Branch $2$ is also selected as the target to add the disruptive disturbance that initiates the chain reaction of cascading blackout. In Fig.~\ref{dcm}, the power network is running in the normal state at Step 1, and the legends are the same as those in Fig.~\ref{cnm}. Then the disruptive disturbance obtained by solving the system of algebraic equation (\ref{con_sys}) (susceptance decrement 10.87) is added to exactly sever Branch 2 at Step 2. Next, Branch 1, Branch 4 and Branch 5 break off simultaneously at Step 3. Afterwards, the DCM reaches a stable state, and it stops the evolution with 2 connected load buses and no power supply.

\begin{figure}[t!]
\scalebox{0.67}[0.67]{\includegraphics{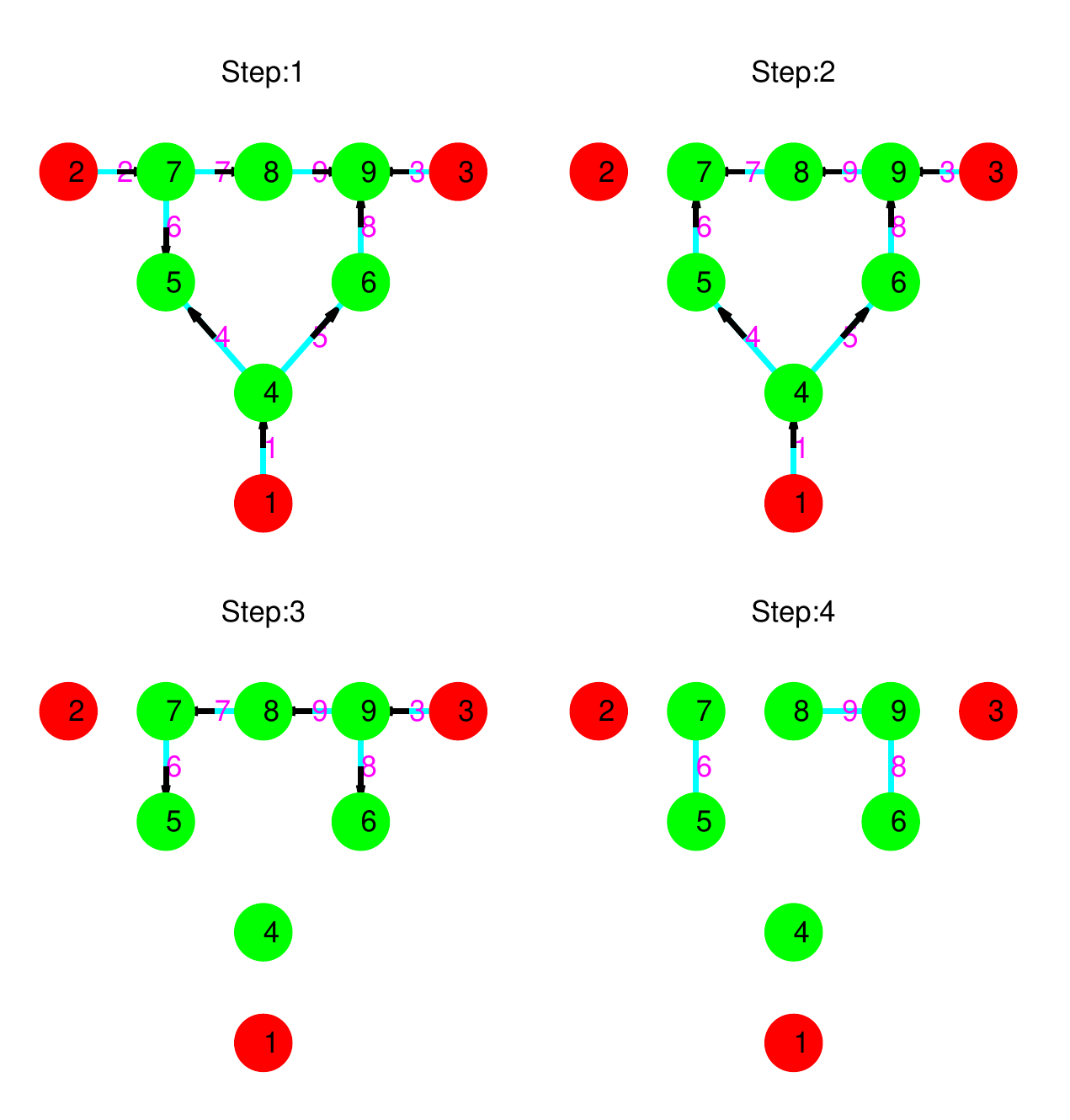}}\centering
\caption{\label{acm} Cascading process of the IEEE 9 bus system based on the ACM.}
\end{figure}

Fig.~\ref{acm} showcases the topology evolution of IEEE 9 bus system based on the ACM. Branch 2 is severed as the initial disturbance, which is the same as that of the DCM. Then we can observe that the ACM evolves like the DCM in the cascading process in terms of both the flow direction and network topology at the first 3 cascading steps. This demonstrates the good approximation of the DCM to the ACM in the cascading process. Finally, the power network includes 3 connected branches at Step 4, which is slightly different from the final configuration of the DCM.

\subsection{Discussions}

\begin{figure}[t!]
\scalebox{0.45}[0.45]{\includegraphics{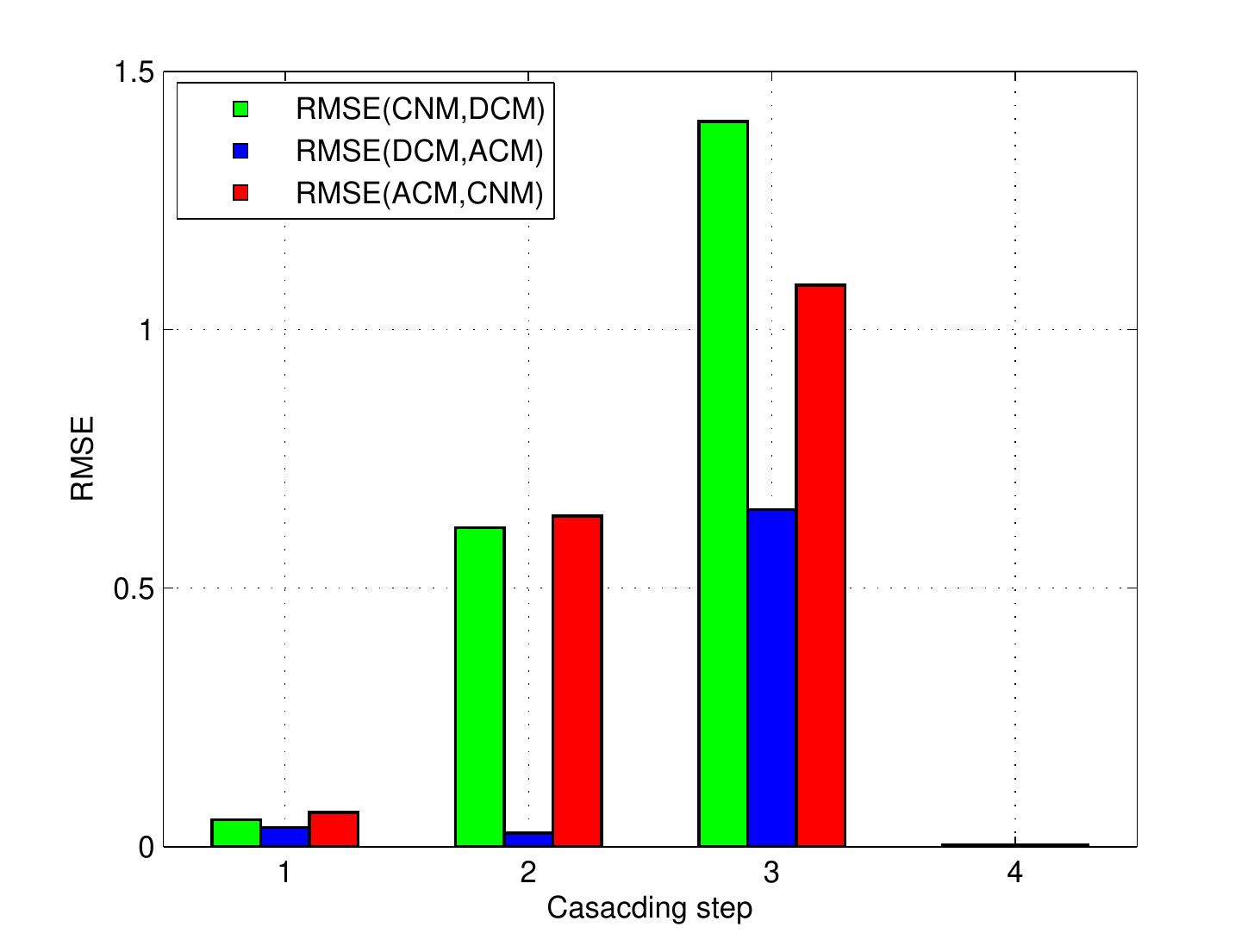}}\centering
\caption{\label{rmse} RMSE of transmission power between two different models at each cascading step.}
\end{figure}

In order to quantify the performance similarity of the three models in the cascading process, we introduce the root mean square error (RMSE) of transmission power on branches, which is defined as
$$
RMSE(X,Y)=\sqrt{\frac{1}{n}\sum_{i=1}^{n}(X_i-Y_i)^2}
$$
where $X,Y\in R^n$ denote the vectors of transmission power on $n$ branches. Intuitively, the smaller RMSE values between two different models imply the more similar performance of these two models at the cascading step. For example, the RMSE values between any two different models at Step 1 are given as: $\texttt{RMSE}(CNM,DCM)=0.0515$, $\texttt{RMSE}(DCM,ACM)=0.0371$ and $\texttt{RMSE}(ACM,CNM)=0.0656$. The above RMSE values indicate that the DCM is quite close to the ACM, and the CNM is relatively closer to the DCM rather than the ACM ($0.0515<0.0656$).

Fig.~\ref{rmse} presents the RMSE of transmission power between two different models at each cascading step. All the RMSE values are quite small at Step 1, which indicates the three models are relatively close to each other in terms of branch power flow. Actually, the three models achieve the same flow direction on each branch according to simulation results in Fig.~\ref{cnm}, Fig.~\ref{dcm} and Fig.~\ref{acm}. At Step 2, the RMSE value between the DCM and the ACM is much smaller than those of the other two model pairs, although the power flows of the three models move in the same direction. All the RMSE values at Step 3 reach the peak, and the RMSE value between the DCM and the ACM is still the smallest among the three model pairs. On the whole, the DCM evolves in the same way that the ACM does under the same initial disturbance ($i.e.$, severing Branch 2). In contrast, the CNM evolves along a totally different trajectory after the outage of Branch 2. Finally, all the three models end up with no transmission power on each branch at Step 4. Thus all the RMSE values become 0.

The above comparisons indicate that the CNM is incompetent to describe the cascading failures of power networks. As for the DCM and the ACM, the two models allow us to obtain the identical flow directions and similar flow magnitudes at Step 1 and Step 2. It is worth pointing out that the numerical algorithm in Matpower failed to converge at Step 3 due to the nonlinearity and non-convexity of the AC power flow equation, which results in the large disparity of power flow between the DCM and the ACM. Thus, the DCM is a good substitute for the ACM when numerical solutions to the AC power flow equation are not available. In short, it is feasible to replace the ACM with the DCM for small scale transmission networks when we investigate cascading failures of power networks.

\section{Conclusion}\label{sec:con}

In this paper, we considered the comparison of three different models that describe the cascading failure of power networks. The critical risks of power networks based on these models were identified with the aid of optimal control theory. Moreover, simulation results on IEEE 9 bus system demonstrated that the model based on pure network topology failed to characterize the actual evolution of cascading blackouts. This indicates that the physical and electrical properties have to be taken into account appropriately even for the cascading dynamics of a simple power network. Future work includes the quantification of cascading path and the risk identification of AC-based power networks.

\section*{Acknowledgment}

This work is partially supported by the Future Resilient Systems Project at the Singapore-ETH Centre (SEC), which is funded by the National Research Foundation of Singapore (NRF) under its Campus for Research Excellence and Technological Enterprise (CREATE) program. It is also partially supported by Ministry of Education of Singapore under contract MOE2016-T2-1-119.

\end{document}